\documentclass[11pt]{35}
\usepackage{amsthm, amsmath, amssymb, amsbsy, amsfonts, latexsym, euscript}
\usepackage{calrsfs} %kalligraphiya
\usepackage{graphicx}

\DeclareMathAlphabet{\varmathbb}{U}{pxsyb}{m}{n}

\def\leq{\leqslant}

\def\geq{\geqslant}
\def\phi{\varphi}

\def\bar{\overline}

\def\kappa{\varkappa}

\newcommand{\D}{\mathrm{d}\kern0.2pt}%
\newcommand{\E}{\mathrm{e}\kern0.2pt} 
\newcommand{\ii}{\kern0.05em\mathrm{i}\kern0.05em}

\newcommand{\RR}{\mathbb{R}}%

\newtheorem{theorem}{\bf \indent Theorem}[section]
\newtheorem{proposition}{\bf \indent Proposition}[section]

\newtheorem{corollary}{\bf \indent Corollary}[section]

\theoremstyle{remark} 
\newtheorem{remark}{\bf \indent Remark}[section]

\newtheorem{definition}{\bf\indent Definition}[section]

\numberwithin{equation}{section}

\begin{document}

\noindent {\Large \bf Mean value properties of solutions to the modified Helmholtz
\\[3pt] equation and related topics (a survey)}

\vskip2mm

{\bf Nikolay Kuznetsov}

\vskip-2pt {\small Laboratory for Mathematical Modelling of Wave Phenomena}
\vskip-4pt {\small Institute for Problems in Mechanical Engineering} \vskip-4pt
{\small Russian Academy of Sciences} \vskip-4pt {\small V.O., Bol'shoy pr. 61, St
Petersburg, 199178} \vskip-4pt {\small Russian Federation} \vskip-4pt {\small
nikolay.g.kuznetsov@gmail.com}

\begin{quote}
\noindent Recent results concerning solutions of the modified Helmholtz equation are
reviewed; namely, various mean value properties and their corollaries, converse and
inverse of these properties, and relations between these solutions and harmonic
functions.
\end{quote}

\vspace{-8mm}

{\centering \section{Introduction} }

\noindent In the present article, we consider real-valued $C^2$-solutions of the
$m$-dimensional modified Helmholtz equation:
\begin{equation}
\nabla^2 u - \mu^2 u = 0 , \quad \mu \in \RR \setminus \{0\} ;
\label{Hh}
\end{equation}
$\nabla = (\partial_1, \dots , \partial_m)$ is the gradient operator, $\partial_i =
\partial / \partial x_i$. Unfortunately, it is not commonly known that these
solutions are called panharmonic functions (or $\mu$-panharmonic functions) by
analogy with solutions of the Laplace equation; this convenient abbreviation coined
by Duffin \cite{D} will be used in what follows. In the latter paper and elsewhere,
Duffin refers to \eqref{Hh} as the Yukawa equation, surprisingly, without citing his
original paper \cite{Y}, in which the Nobel Prize winning theory of nuclear forces
was proposed. For describing the force potential of a point charge that decays
rapidly at infinity, Yukawa used the three-dimensional fundamental solution of
\eqref{Hh}, which has this property. Thus, it is quite reasonable to name equation
\eqref{Hh} after Yukawa, even taking into account, that he did not use it in
\cite{Y}. However, one can find still more confusing name of \eqref{Hh}, namely, the
Helmholtz equation; see \cite{S} and \cite{LNV}, p.~231.
 
Undeservedly, panharmonic functions received much less attention than harmonic and
subharmonic, despite the fact that studies of the three-dimensional equation
\eqref{Hh} were initiated by C.~Neumann in his monograph \cite{NC} published in
1896. Since then, ridiculously small number of papers treating rigorously solutions
of equation \eqref{Hh} have been published and their content varies noticeably. Some
consider particular boundary value problems (see, for example, \cite{BD},
\cite{BrDe} and \cite{S}), whereas others are concerned with the so-called
$\mu$-regular class of pseudoanalytic functions satisfying the Cauchy--Riemann
equations for the two-dimensional version of \eqref{Hh}; see, for example, \cite{D},
\cite{D1} and \cite{SW}. Finally, it is worth mentioning the representation formulas
for panharmonic functions obtained in \cite{CL}.

Recently, the author published several notes \cite{Ku1}, \cite{Ku2}, \cite{Ku3},
\cite{Ku4}, \cite{Ku5}, \cite{Ku6} and \cite{Ku7} dealing with various topics: mean
value formulae for panharmonic functions, their corollaries, converse theorems and
other related results such as characterization of balls via these functions. These
studies were initiated after completing the survey \cite{Ku}, preparing which it was
a surprise to discover that only mean value formulae for spheres and circumferences
in three and two dimensions were derived earlier by C.~Neumann and Duffin,
respectively. Moreover, relations between harmonic, subharmonic and panharmonic
functions also remained unnoticed.

The aim of this survey is to present the obtained results in a systematic,
self-contained form similar, to some extent, to a comprehensive theory developed by
Duffin for two-dimensional punharmonic functions; see \cite{D} and \cite{D1}. The
plan of the article is as follows. Various mean value formulae (volume, spherical,
asymptotic etc.) and their corollaries are presented in Sect.~2, whereas Sect.~3
deals with converse of these mean value properties. Two characterizations of balls
via panharmonic functions are described in Sect.~4. Relations between harmonic and
panharmonic functions are considered in Sect.~5.

\vspace{2mm}

{\centering \section{Mean value formulae and their corollaries} }

The following analogue of the Gauss theorem on the arithmetic mean of a harmonic
function over an $m-1$-dimensional sphere in ${\RR}^m$ (the original memoir
\cite{G}, Article~20, deals with the case $m=3$) was derived for panharmonic
functions in \cite{Ku1}.

\begin{theorem}[\cite{Ku1}]
Let $u$ be panharmonic in a domain $D \subset {\RR}^m$, $m \geq 2$. Then for every
$x \in D$ the identity\\[-3mm]
\begin{equation}
M^\circ (u, x, r) =  a^\circ (\mu r) \, u (x) , \quad a^\circ (\mu r) = \Gamma
\left( \frac{m}{2} \right) \frac{I_{(m-2)/2} (\mu r)}{(\mu r / 2)^{(m-2)/2}} \, ,
\label{MM}
\end{equation}
holds for each admissible sphere $S_r (x);$ $I_\nu$ denotes the modified Bessel
function of order $\nu$.
\end{theorem}

Here and below the following notation and terminology are used. The open ball of
radius $r$ centred at $x$ is denoted by $B_r (x) = \{ y : |y-x| < r \}$; the latter
is called admissible with respect to a domain $D$ provided $\overline{B_r (x)}
\subset D$, and $S_r (x) = \partial B_r (x)$ is the corresponding admissible sphere.
If $u \in C^0 (D)$, then its spherical mean value over $S_r (x) \subset D$
is\\[-2mm]
\begin{equation}
M^\circ (u, x, r) = \frac{1}{|S_r|} \int_{S_r (x)} u (y) \, \D S_y =
\frac{1}{\omega_m} \int_{S_1 (0)} u (x +r y) \, \D S_y \, , \label{sm}
\end{equation}
where $|S_r| = \omega_m r^{m-1}$ and $\omega_m = 2 \, \pi^{m/2} / \Gamma (m/2)$ is
the total area of the unit sphere (as usual $\Gamma$ stands for the Gamma function),
and $\D S$ is the surface area measure. It is clear that this function is continuous
in $x$ and $r$; moreover, if $u \in C^k (D)$, then its mean value is in the same
class in $x$ and $r$. By continuity we have that\\[-2mm]
\[ M^\circ (u, x, 0) = u (x) ,
\]
whereas further identities for $M^\circ$ can be found in \cite{J}, Chapter IV.

For $m=3$ identity \eqref{MM} (derived by C.~Neumann \cite{NC} as early as 1896) is
particularly simple because $a^\circ (\mu r) = \sinh (\mu r) / (\mu r)$. Duffin
independently rediscovered the proof (see~\cite{D}, pp.~111-112), but in two
dimensions with $a^\circ (\mu r) = I_0 (\mu r)$.

Our proof of Theorem 2.1 is based on the Euler--Poisson--Darboux equation for
$M^\circ$ (see \cite{J}, p.~88):
\begin{equation}
M^\circ_{rr} + (m-1) \, r^{-1} M^\circ_r = \nabla^2_x M^\circ \, , \ \ \mbox{where}
\ r > 0 , \ x \in D \, . \label{EPD}
\end{equation}
It is valid provided $u \in C^2 (D)$ and follows from the obvious relation
\begin{equation}
m \int_0^r t^{m-1} M^\circ (u, x, t) \, \D t = r^m  M^\bullet (u, x, r) \, ,
\label{M+M}
\end{equation}
where
\[ M^\bullet (u, x, r) = \frac{1}{|B_r|} \int_{B_r (x)} \!\! u (y) \, \D y 
\]
and $|B_r| = \omega_m r^m / m$ is the volume of $B_r$. Indeed, applying the
Laplacian to both sides of \eqref{M+M}, we obtain
\[ \omega_m \int_0^r t^{m-1} \nabla^2_x M^\circ (u, x, t) \, \D t = 
\int_{B_r (0)} \!\! \nabla^2_x u (x + y) \, \D y \, .
\]
By Green's first formula the last integral is equal to
\[ \int_{|y| = r} \!\! \nabla_x \, u (x+y) \cdot \frac{y}{r} \, \D S_y \, ,
\]
and changing variables this can be written as follows:
\[ r^{m-1} \frac{\partial}{\partial r} \int_{|y|=1} \!\! u (x + r y) \, 
\D S_y = \omega_m r^{m-1} M_r^\circ (u, x, r) \, .
\]
Thus we arrive at the equality
\begin{equation*}
r^{m-1} M_r^\circ (u, x, r) = \int_0^r t^{m-1} \nabla^2_x M^\circ (u, x, t) \, \D t
\label{M_r}
\end{equation*}
Differentiation of this relation with respect to $r$ yields \eqref{EPD}.

Now we are in a position to prove the theorem.

\begin{proof}[Proof of Theorem 2.1.]
It is straightforward to show that $a (r) = a^\circ (\mu r)$ is a unique solution
the following Cauchy problem:
\begin{equation}
a_{rr} + (m-1) \, r^{-1} a_r - \mu^2 a = 0 , \ \ a (0) = 1 , \ \ a_r (0) = 0 \, .
\label{g5}
\end{equation}
This follows by virtue of the relations (see \cite{Wa}, p. 79):
\begin{equation}
z I_{\nu+1} (z) + 2 \nu I_\nu (z) - z I_{\nu-1} (z) = 0 \, , \ \ \ [z^{-\nu} I_\nu
(z)]' = z^{-\nu} I_{\nu+1} (z) \, . \label{diff}
\end{equation}
In particular, the second one implies the second initial condition.

The function $w (r, x) = a^\circ (\mu r) \, u (x) - M^\circ (u, x, r)$ is defined
for all $x \in D$ and all $r \geq 0$ such that $S_r (x)$ is admissible and satisfies
the initial conditions
\[ w (x, 0) = 0 \, , \quad w_r (x, 0) = 0 \, .
\]
The first one is a consequence of the identities $a^\circ (0) = 1$ and $M^\circ (u,
x, 0) = u (x)$, whereas the second one follows from equation \eqref{EPD} multiplied
by $r$ in the limit as $r \to 0$. Moreover, equations \eqref{EPD} and \eqref{g5}
yield that
\[ w_{rr} + (m-1) \, r^{-1} w_r - \mu^2 w = 0 \ \ \mbox{for} \ r > 0 \, .
\]
Since the latter Cauchy problem has only a trivial solution, we obtain \eqref{MM}.
\end{proof}

It follows from Theorem 2.1 that a panharmonic function of fixed sign belongs to one
of two well studied classes of functions; namely, subharmonic or superharmonic. In
our context, it is sufficient to define these classes as follows.

\begin{definition}[Gilbarg and Trudinger \cite{GT}, p.~23]
A function $u \in C^0 (D)$ is called subharmonic (superharmonic) in $D$ if for every
admissible ball $B \subset D$ and every function $h$ harmonic in $B$ and satisfying
$u \leq h$ ($u \geq h$) on $\partial B$, the same inequality holds throughout~$\bar
B$.
\end{definition}

\begin{corollary}[\cite{Ku7}]
Let a panharmonic function $u$ be nonnegative (nonpositive) in a domain $D$. Then
$u$ is subharmonic (superharmonic) in $D$.
\end{corollary}

\begin{proof}
The function $a^\circ$ increases monotonically on $[0, \infty)$ from $a^\circ (0) =
1$ to infinity. Indeed, the second relation \eqref{diff} implies the monotonicity,
whereas the behavior at infinity is a consequence of the asymptotic formula
\begin{equation}
I_\nu (z) = \frac{\E^z}{\sqrt{2 \pi z}} \left[ 1 + O (|z|^{-1}) \right] \, , \ \
|\arg z| < \pi /2 , \label{asym_I}
\end{equation}
whose principal term does not depend on $\nu$, is valid as $|z| \to \infty$; see
\cite{Wa}, p. 80.

Since $a^\circ (\mu r) > 1$, identity \eqref{MM} yields that $u (x) \leq M^\circ (u,
x, r)$ for every $S_r (x) \subset D$ provided the $\mu$-panharmonic $u$ is
nonnegative in $D$. The result immediately follows from this inequality and
Definition 2.1.
\end{proof}

The converse of Corollary 2.1 is not true, because any nonzero constant is
subharmonic and superharmonic, but not panharmonic. Another consequence of Theorem
2.1 is the following.

\begin{corollary}[\cite{Ku1}]
Let $D$ be a domain in ${\RR}^m$, $m \geq 2$. If $u$ is panharmonic in $D$, then
\begin{equation}
M^\bullet (u, x, r) =  a^\bullet (\mu r) \, u (x) , \quad a^\bullet (\mu r) =
\Gamma \left( \frac{m}{2} + 1 \right) \frac{I_{m/2} (\mu r)}{(\mu r / 2)^{m/2}} \, ,
\label{MM'}
\end{equation}
and
\begin{equation} 
a^\circ (\mu r) M^\bullet (u, x, r) = a^\bullet (\mu r) M^\circ (u, x, r) 
\label{MMtil}
\end{equation}
for every admissible ball $B_r (x)$.
\end{corollary}

\begin{proof}
Let us write formula \eqref{MM} in the form
\begin{equation}
\omega_m^{-1} \int_{S_1 (0)} \!\! u (x + \rho y) \, \D S_y = a^\circ (\mu \rho) \, u
(x) \, , \label{03Au}
\end{equation}
multiply by $\rho^{m-1}$, and integrate with respect to $\rho$ over $(0, r)$, where
$r$ is such that $B_r (x)$ is admissible. Thus we obtain $M^\bullet (u, x, r)$ on
the left-hand side after division by $r^m$. Applying formula 1.11.1.5, \cite{PBM},
namely,
\begin{equation*}
\int_0^x \!\! x^{1 + \nu} I_\nu (x) \, \D x = x^{1 + \nu} I_{\nu + 1} (x) \, , \ \
\Re \, \nu > -1 . \label{PBM}
\end{equation*}
with $\nu = (m-2)/2$ while integrating the right-hand side, identity \eqref{MM'}
follows. Combining \eqref{MM} and \eqref{MM'}, one arrives at \eqref{MMtil}.
\end{proof}

\begin{remark}
It is clear from the proof of Corollary 2.2 that identities \eqref{MM} and
\eqref{MM'} are equivalent. Like $a^\circ$, the function $a^\bullet$ increases
monotonically on $[0, \infty)$ from $a^\bullet (0)=~1$ to infinity. Moreover,
$a^\bullet (t) / a^\circ (t) < 1$ for $t > 0$, which immediately follows from their
definitions and the first formula \eqref{diff}.

Identity \eqref{MMtil} couples the mean values over spheres and balls for a
$\mu$-panharmo\-nic $u$. In this identity, the ratio of coefficients $a^\circ (\mu
r) / a^\bullet (\mu r)$ tends to unity in the limit as $\mu \to 0$, thus reducing
\eqref{MMtil} to the formula equating the mean values of a harmonic function over
spheres and balls.
\end{remark}

Another corollary of Theorem 2.1 deals with iterated spherical means introduced by
John; see \cite{J}, p.~78, but the notation is different here. Let $D_r$ be the
subdomain of $D$ with boundary `parallel' to $\partial D$ at the distance $r > 0$;
namely, $D_r = \{ x \in D : \overline{B_r (x)} \subset D \}$. Thus, $D_r$ is
nonempty only when $r$ is less than the radius of the open ball inscribed into~$D$.
Since $M^\circ (u, \cdot, r)$ is defined on $D_r$, it is clear that
\begin{equation*}
I (u, x, r', r) = M^\circ (M^\circ (u, \cdot, r), x, r') = \frac{1}{\omega_m}
\int_{S_1 (0)} M^\circ (u, x + r' y, r) \, \D S_y \label{im}
\end{equation*}
is defined on $D_{r'+r}$; here the second equality is a consequence of \eqref{sm}.
Substituting the expression for $M^\circ$, we obtain
\begin{equation}
I (u, x, r', r) = \frac{1}{\omega_m^2} \int_{S_1 (0)} \int_{S_1 (0)} u (x + r' y +
r z) \, \D S_z \, \D S_y \, , \label{im'}
\end{equation}
and so it is symmetric in $r'$ and $r$, that is, $I (u, x, r', r) = I (u, x,
r, r')$. Moreover,
\[ I (u, x, 0, r) = I (u, x, r, 0) = M^\circ (u, x, r) \quad \mbox{and} \quad I 
(u, x, 0, 0) = u (x) .
\]

Now we see that \eqref{03Au} implies the following assertion concerning the
iterated mean value property.

\begin{corollary}
Let $u$ be $\mu$-panharmonic in a domain $D \subset \RR^m$, $m \geq 2$. If the
domain $D_r$ is nonempty for $r > 0$, then $M^\circ (u, \cdot, r)$ is
$\mu$-panharmonic in it and
\begin{equation*}
I (u, x, r', r) =  a^\circ (\mu r') \, M^\circ (u, x, r) = a^\circ (\mu r') \,
a^\circ (\mu r) \, u (x) \label{iter}
\end{equation*}
for every $x \in D_r$ and all $S_{r'} (x)$ admissible with respect to $D_r$.
\end{corollary}

Two more mean value properties of panharmonic functions are analogous to the
classical theorems of Blaschke \cite{Bla}, Priwaloff \cite{Pri} and Zaremba
\cite{Za} concerning asymptotic mean values of harmonic functions (see \cite{NV},
Sect.~9, for a discussion).

\begin{proposition}
Let $D$ be a domain in $\RR^m$, $m \geq 2$. If $u$ is panharmonic in $D$, then
\begin{equation}
\lim_{r \to +0} \frac{M^\bullet (u, x, r) - u (x)}{r^2} = \frac{\mu^2 u (x)} {2
(m+2)} \quad \mbox{for every} \ x \in D . \label{par}
\end{equation}
The assertion also holds with $M^\bullet (u, x, r)$ changed to $M^\circ (u, x, r)$
and the right-hand side term in \eqref{par} changed to $\mu^2 u (x) / (2 m)$.
\end{proposition}

\begin{proof}
The well-known relationships between the Laplacian and asymptotic mean values (see
\cite{Br}, Ch.~2, Sect.~2) follow from Taylor's formula:
\[ u (x + y) - u (x) = y \cdot \nabla u (x) + 2^{-1} y \cdot [H_u (x)] y + o (r^2) \, .
\]
It is valid for $u \in C^2 (D)$ at $x \in D$ as $r \to 0$ provided $B_r (x)$ is
admissible and $|y| \leq r$; here $H_u (x)$ denotes the Hessian matrix of $u$ at $x$
and ``$\cdot$'' stands for the inner product in $\RR^m$. Averaging each term of the
equality with respect to $y \in B_r (0)$, one obtains
\[ M^\bullet (u, x, r) - u (x) = \frac{1}{2 |B_r|} \int_{B_r (0)} y \cdot [H_u (x)] 
y \, \D y + o (r^2) \, ,
\]
because the mean value of the first order term vanishes. It is straightforward to
calculate that
\begin{equation}
\lim_{r \to +0} \frac{M^\bullet (u, x, r) - u (x)}{r^2} = \frac{\nabla^2 u (x)} {2
(m+2)} \, , \quad x \in D . \label{vol}
\end{equation}
Now \eqref{par} follows from panharmonicity of $u$.

Similarly, averaging Taylor's formula with respect to $y \in S_r (0)$, we obtain
\begin{equation}
\lim_{r \to +0} \frac{M^\circ (u, x, r) - u (x)}{r^2} = \frac{\nabla^2 u (x)}{2 m}
\, , \quad x \in D . \label{sph}
\end{equation}
This and panharmonicity of $u$ yield the second assertion.
\end{proof}

\begin{remark}
Another proof of Proposition 2.1 is as follows. Identity \eqref{MM} holds for a
panharmonic $u$ provided $B_r (x)$ is admissible. It implies that \eqref{par} is
equivalent to
\[ \lim_{r \to +0} \frac{a^\bullet (\mu r) - 1}{(\mu r)^2} = \frac{1}{2 (m+2)} \, ,
\]
which follows from the definition of $I_{m/2}$.

In order to obtain the second assertion of Proposition 2.1 one has to use the
equality
\[ \lim_{r \to +0} \frac{a^\circ (\mu r) - 1}{(\mu r)^2} = \frac{1}{2 m} \, ,
\]
which is true by the definition of $I_{(m-2)/2}$.
\end{remark}

\subsection{Applications of Theorem 2.1}

Turning to applications of the obtained mean value property, we recall that the most
important consequences of the corresponding property in the case of harmonic
functions are the strong maximum principle and Liouville's theorem. The first
asserts that a function harmonic in a domain $D$ cannot have local maxima or minima
there; moreover, if it is continuous in $\overline D$, which is bounded, then its
maximum and minimum are attained on $\partial D$. The second theorem says that every
harmonic on $\RR^m$ function bounded below (or above) is constant.

It is clear that $u (x) = (\mu \, |x|)^{-1} \sinh (\mu \, |x|)$, which is
panharmonic in $\RR^3$, violates both these assertions; indeed, it has the local
(and global) minimum at the origin. Since the maximum principle and Liouville's
theorem, as formulated above, are not true for panharmonic functions, some extra
restrictions must be imposed in order to convert these theorems into true ones.

\vspace{2mm}

\noindent {\bf 2.1.1. The weak maximum principle.} We begin with the following
assertion concerning the behaviour of $|u|$ for a nontrivial panharmonic function
$u$.

\begin{proposition}
Let $u$ be a nonvanishing identically panharmonic function in a do\-main $D \subset
\RR^m$, $m \geq 2$. Then for every $x \in D$ there exists $y \in D$ such that $|u
(y)| > |u (x)|$.
\end{proposition}

\begin{proof}
Without loss of generality, we assume that $u (x) \geq 0$; indeed, $-u$ should be
con\-sidered otherwise. Then Theorem~2.1 implies that $M^\circ (u, x, r) \geq 0$ for
every admis\-sible sphere $S_r (x)$ and $u (x) < M^\circ (u, x, r)$ because $a^\circ
(\mu r) > 1$. Therefore, there exists a point $y \in S_r (x) \subset D$ such that $u
(y) > u (x)$.
\end{proof}

An immediate consequence of this proposition is the weak maximum principle for
panharmonic functions.

\begin{theorem}
Let $D$ be a bounded domain in $\RR^m$, $m \geq 2$. If $u \in C^0 (\overline D)$ is
panharmonic in $D$, then
\begin{equation}
\sup_{x \in D} | u (x) | = \max_{x \in \partial D} | u (x) | \, . \label{wmp}
\end{equation}
\end{theorem}

\begin{proof}
In the case of $u$ nonvanishing identically, we take a sequence
$\{x_k\}_{k=1}^\infty \subset D$ such that
\[ |u (x_k)| \to \sup_{x \in D} | u (x) | \ \ \mbox{as} \ k \to \infty .
\]
Since $D$ is bounded, $\{x_k\}_{k=1}^\infty$ has a limit point in $\overline D$,
say $x_0$, and $|u (x_0)| = \sup_{x \in D} | u (x) |$ by continuity. Moreover, $x_0
\in \partial D$; indeed, if $x_0 \in D$, then there exists $y \in D$ such that $|u
(y)| > |u (x_0)|$ by Proposition~2.2, but this is impossible. Now \eqref{wmp}
follows from the equality $\sup_{x \in D} | u (x) | = \max_{x \in \overline D} | u
(x) |$ valid becaus $u \in C^0 (\overline D)$.
\end{proof}

Here, only the mean value property is used for proving this principle for
panharmonic functions. Of course, the approach to proving this principle that is
applicable to general elliptic equations (see \cite{GT}, Sect.~3.1) is valid for
these functions as well, but our aim was to use a minimal tool. An immediate
consequence of Theorem~2.2 (see \cite{NC}, p.~260, for the original formulation) is
the uniqueness of a solution to the Dirichlet problem for equation \eqref{Hh} in a
bounded domain as well as the continuous dependence of solutions to this problem on
boundary data.

\vspace{2mm}

\noindent {\bf 2.1.2. Liouville's theorem.} In the whole $\RR^m$, self-similarity
allows us to restrict ourselves to the equation:
\begin{equation}
\nabla^2 u - u = 0 \, ; \label{Hh1}
\end{equation}
indeed, it follows from \eqref{Hh} by a change of variables.

\begin{theorem}
Let $u$ be a solution of \eqref{Hh1} on $\RR^m$. If the inequality
\begin{equation}
| u (x) | \leq C (1 + |x|)^n \ \ holds \ for \ all \ x \in \RR^m \label{N}
\end{equation}
with some $C > 0$ and a nonnegative integer $n$, then $u$ vanishes identically.
\end{theorem}

\begin{proof}
Substituting the asymptotic formula \eqref{asym_I} in the expression for $a^\circ
(r)$, we obtain
\[ a^\circ (r) = \frac{\Gamma (m / 2) \, 2^{(m-3)/2}}{\sqrt \pi}
\frac{\E^r}{r^{(m-1)/2}} \left[ 1 + O (r^{-1}) \right] \ \ \mbox{as} \ r \to \infty \, .
\]
Using this and \eqref{N} in identity \eqref{03Au} with $\mu = 1$, we see that the
inequality
\[ | u (x) | \leq \widetilde C (1 + |x| + r)^n \frac{r^{(m-1)/2}}{\E^r} 
\]
holds with some $\widetilde C > 0$ for all $x \in \RR^m$ and all $r > 0$. Letting $r
\to \infty$, the required assertion follows.
\end{proof}

Inequality \eqref{N} with any (arbitrarily large) $n > 0$ implies that a solution of
equation \eqref{Hh1} is trivial. On the other hand, if the same inequality is
imposed on a harmonic function, then it is a (harmonic) polynomial, whose degree is
less than or equal to $n$; see \cite{Vl}, p.~290.

{\centering \section{Converse of mean value properties} }

In the classical Kellogg's monograph \cite{K}, the section, that follows the proof
of the Gauss theorem on the spherical arithmetic means for harmonic functions,
begins with the sentence.
\begin{quote}
The property of harmonic functions given by Gauss' theorem is so simple and
striking, that it is of interest to inquire what properties functions have which
are, as we shall express it, their own arithmetic means on the surface of spheres.
\end{quote}
Then Kellogg proves the converse of the arithmetic mean theorem due to Koebe (1906),
and adds: ``This theorem will be of repeated use to us.'' Its analogue for
panharmonic functions was obtained only in 2021.

\begin{theorem}[\cite{Ku1}]
Let $D$ be a bounded domain in $\RR^m$. If identity \eqref{MM} with $\mu > 0$ is
fulfilled for $u \in C^0 (D)$ at every $x \in D$ and for all $r \in (0, r (x))$,
where $B_{r (x)} (x)$ is admissible, then $u$ is $\mu$-panharmonic in $D$.

If instead of \eqref{MM} identity \eqref{MM'} is fulfilled for $u \in C^0 (D)$ in
the same way as above, then $u$ is $\mu$-panharmonic in $D$.
\end{theorem}

\begin{proof}
First, we have to show that $u$ is smooth for which purpose a trick applied by
Mikhlin in his proof of Koebe's theorem is helpful; see \cite{M}, Ch.~11, Sect.~7.
It is based on using the mollifier $\omega_\epsilon (|y - x|) = \omega_\epsilon
(r)$; see its properties in \cite{M}, Ch.~1, Sect.~1.

Let $\epsilon > 0$ be small and let $D' = D_{2 \epsilon}$ (the definition of this
domain is given prior to Corollary~2.3). Assuming that $x \in D'$, we multiply
\eqref{MM} by $\omega_\epsilon$, thus obtaining
\[ u (x) \, a^\circ (\mu r) \, |S_r| \, \omega_\epsilon (r) = 
\omega_\epsilon (r) \int_{S_r (x)} \!\! u (y) \, \D S_y \, .
\]
Now, integration with respect to $r$ over $(0, \epsilon)$ yields
\[ u (x) \, c (\mu, \epsilon) = \int_{B_\epsilon (x)} \!\! u (y) \, 
\omega_\epsilon (|y - x|) \, \D y = \int_{D} \!\! u (y) \, \omega_\epsilon (|y -
x|) \, \D y \, .
\]
Here the last equality is valid because $x \in D'$, whereas $\omega_\epsilon (|y -
x|)$ vanishes outside $B_\epsilon (x)$. Also,
\[ c (\mu, \epsilon) = \int_0^\epsilon a^\circ (\mu r) \, |S_r| \, \omega_\epsilon
(r) \, \D r > 0 \, ,
\]
because $a^\circ (\mu r) > 1$. Since $\omega_\epsilon$ is infinitely differentiable,
the obtained representation shows that $u \in C^\infty (D')$. However, $\epsilon$ is
arbitrarily small, and so $u \in C^\infty (D)$.

Now we are in a position to demonstrate that $u$ is panharmonic in $D$. Since
\eqref{MM} holds for every $x \in D$ and all $r \in (0, r (x))$ provided $B_{r (x)}
(x)$ is admissible, identity \eqref{MM'} holds as well. Applying the Laplacian to
the integral on the left-hand side of the latter identity, we obtain
\[ \int_{|y| < r} \!\! \nabla^2_x \, u (x+y) \, \D y = \int_{|y| = r} \!\! \nabla_x \,
u (x+y) \cdot \frac{y}{r} \, \D S_y \, .
\]
Here the equality is a consequence of Green's first formula. By changing variables
this can be written as follows:
\[ r^{m-1} \frac{\partial}{\partial r} \int_{|y|=1} \!\! u (x + r y) \, 
\D S_y = |S_1 (0)| r^{m-1} \frac{\partial}{\partial r} M^\circ (u, x, r) \, .
\]
However, $M^\circ (u, x, r) = a^\circ (\mu r) \, u (x)$, and so the second formula
\eqref{diff} yields that
\[ \frac{\partial}{\partial r} M^\circ (u, x, r) = - \frac{\mu I_{m/2} (\mu r)}
{(\mu r / 2)^{(m-2)/2}} \, u (x) \, .
\]
Combining the above considerations and \eqref{MM'}, we find that for every $x \in D$
the equality
\[ \int_{|y| < r} \!\! [ \nabla^2_x \, u - \mu^2 u ] \, (x+y) \, \D y = 0 \ \ 
\mbox{holds for all} \ r \in (0, r (x)) \, .
\]
Thus, every ball $B_r (x)$ contains a point $y (r, x)$ such that $[ \nabla^2 \, u -
\mu^2 u ] \, (y (r, x)) = 0$. Since $y (r, x) \to x$ as $r \to 0$, we conclude by
continuity that $u$ satisfies equation \eqref{Hh} at every $x \in D$, that is, $u$
is panharmonic in $D$.

Let identity \eqref{MM'} hold for $u$ instead of \eqref{MM}. Since these two
identities are equivalent (see Remark~2.1), \eqref{MM} holds for $u$ as well. Then
the previous considerations yield the assertion.
\end{proof}

\begin{corollary}
Let $D \subset \RR^m$, $m \geq 2$, be a bounded domain, and let $u \in C^0 (D)$. If
for every $x \in D$ there exists $r (x)$ such that $B_{r (x)} (x)$ is admissible and
$M^\circ (u,x,r) / a^\circ (\mu r)$ does not depend on $r \in (0, r (x))$, then $u$
is $\mu$-panharmonic in $D$.
\end{corollary}

\begin{proof}
According to the mean value theorem for integrals, for every $x \in D$ and each $r
\in (0, r (x))$ there exists $x_0 (r) \in S_r (x)$ such that $ M^\circ (u,x,r) = u
(x_0 (r))$, and so
\[ M^\circ (u,x,r) / a^\circ (\mu r) = u (x_0 (r)) / a^\circ (\mu r) \, .
\]
Since this continuous function of $r$ is constant on $(0, r (x))$, it is equal to
its limit as $r \to 0$. Since $a^\circ (\mu r) \to 1$ and $u (x_0 (r)) \to u (x)$ as
$r \to 0$, we obtain that \eqref{MM} holds for every $x \in D$ and all $r \in (0, r
(x))$. Then Theorem~3.1 yields the assertion.
\end{proof}

Now, we prove the converse of identity \eqref{MMtil}, which generalizes the result
of Beckenbach and Reade \cite{BR} for harmonic functions; it was announced in
\cite{Ku4} without proof.

\begin{theorem}
Let $D \subset {\RR}^m$, $m \geq 2$, be a bounded domain, and let $u \in C^0 (D)$.
If identity \eqref{MMtil} holds for every $x \in D$ and all $r \in (0, r (x))$,
where $r (x) > 0$ is such that the ball $B_{r (x)} (x)$ is admissible, then $u$ is
$\mu$-panharmonic in~$D$.
\end{theorem}

\begin{proof}
Let $\rho > 0$ be sufficiently small. If $r \in (0, \rho)$, then $M^\bullet (x,r,u)$
is defined for every $x$, which belongs to an open subset of $D$ depending on the
smallness of $\rho$. Moreover, $M^\bullet (x,r,u)$ is differentiable with respect to
$r$ and
\[ \partial M^\bullet (x,r,u) / \partial r = m r^{-1} [ M^\circ (x,r,u) -
M^\bullet (x,r,u) ] \quad \mbox{for} \ r \in (0, \rho) .
\]
Since
\begin{equation*}
\frac{a^\bullet (\mu r)}{a^\circ (\mu r)} = \frac{m I_{m/2} (\mu r)}{\mu r
I_{(m-2)/2} (\mu r)} \, , \label{18}
\end{equation*}
the previous relation takes the form
\[ \frac{\partial M^\bullet / \partial r}{M^\bullet} = \mu \frac{I_{(m-2)/2} (\mu r)}
{I_{m/2} (\mu r)} - \frac{m}{r} = \mu \frac{I_{m/2}' (\mu r)} {I_{m/2} (\mu r)} -
\frac{m}{2 r} \, ,
\]
where the last equality is a consequence of the recurrence formula (\cite{Wa},
p.~79):
\[ I_{\nu-1} (z) = I_{\nu}' (z) + \frac{\nu}{z} I_{\nu} (z) \, .
\]

The equation for $M^\bullet$ has logarithmic derivatives on both sides. Therefore,
integrating with respect to $r$ over the interval $(\epsilon, \rho)$, we obtain,
after letting $\epsilon \to 0$, relation \eqref{MM'} with $r$ changed to $\rho$.
Indeed, shrinking $B_\epsilon (x)$ to its centre on the left-hand side, we see that
$M^\bullet (x,\epsilon,u) \to u (x)$ because $u \in C^0 (D)$, and this takes place
for every $x$ in an arbitrary closed subset of $D$. By letting $\epsilon \to 0$ on
the right-hand side, the factor $\Gamma \left( \frac{m}{2} + 1 \right)$ arises due
to the leading term of the power expansion of $I_{m/2}$. Thus we have
\[ M^\bullet (x,\rho,u) = a^\bullet (\mu \rho) \, u (x)
\]
for every $x \in D$ and all admissible $\rho$. Hence $u$ is $\mu$-panharmonic in $D$
by the second assertion of Theorem~3.1.
\end{proof}

\subsection{Restricted mean value property}

There is a long series of publications dealing with the so-called restricted mean
value properties that characterize harmonicity; see the survey article \cite{NV},
Sections~5 and 6. The following definition is accommodated for panharmonic
functions.

\begin{definition}
A real-valued function $f$ defined on an open set $G \subset \RR^m$ is said to have
the restricted mean value property with respect to spheres if for each $x \in G$
there exists a single sphere centred at $x$ of radius $r (x)$ such that $B_{r (x)}
(x) \subset G$ and identity \eqref{MM} holds for $f$ with $r = r (x)$.
\end{definition}

%Now we are in a position to prove the following.

\begin{theorem}
Let $D \subset \RR^m$, $m \geq 2$, be a bounded domain such that the Dirichlet
problem for equation \eqref{Hh} is soluble in $C^2 (D) \cap C^0 (\overline D)$ for
every continuous function given on $\partial D$. If $u \in C^0 (\overline D)$ has
the restricted mean value property in $D$ with respect to spheres, then $u$ is
$\mu$-panharmonic in $D$.
\end{theorem}

\vspace{-2mm}

\begin{proof}
First, let us show that the theorem's assumptions yield that\\[-3mm]
\begin{equation}
\max_{x \in \overline D} |u (x)| = \max_{x \in \partial D} |u (x)| \, . \label{max}
\end{equation}
Reasoning by analogy with the proof of Proposition~2.2, we see that the restricted
mean value property implies that for every $x \in D$ there exists $y \in D$ such
that $|u (y)| > |u (x)|$. Then the considerations used in the proof of Theorem~2.2
yield \eqref{max}.

Let $f$ denote the trace of $v$ on $\partial D$; then there exists $u_0 \in C^0
(\overline D)$ solving the Dirichlet problem for equation \eqref{Hh} in $D$ with $f$
as the boundary data. Hence $u_0$ satisfies identy \eqref{MM} for all $x \in D$ and
all admissible $S_r (x)$, and so the restricted mean value property is valid for $u
- u_0$. Then the weak maximum principle \eqref{max} holds for $u - u_0$, thus
implying that $u \equiv u_0$ in $D$ because $u \equiv u_0$ on $\partial D$. Then $u$
also satisfies \eqref{MM} for every $x \in D$ and all admissible $S_r (x)$. Now,
Theorem~3.1 yields that $u$ is panharmonic in $D$.
\end{proof}

The question about domains in which the Dirichlet problem for an elliptic equation
is soluble has a long history going back to George Green's \textit{Essay on the
Application of Mathematical Analysis to the Theories of Electricity and Magnetism}
published in 1828, where this problem for the Laplace equation was posed for the
first time. The final answer when the Dirichlet problem for harmonic functions has
a solution was given by Wiener \cite{NW} in 1924; the notion of capacity was
introduced for this purpose.

The class of bounded domains such that the Dirichlet problem is soluble is the same
for the modified Helmholtz equation and for the Laplace equation. This follows from
the results of Oleinik \cite{O} and Tautz \cite{Ta}; they demonstrated independently
and published in 1949 that this fact about the solubility of the Dirichlet problem
is a common characteristic which is true for elliptic equations of rather general
form (see the monograph \cite{Mi}, Ch.~IV, Sect.~28, for a review of related
papers).

\vspace{-2mm}

\subsection{A function with panharmonic means is panharmonic itself}

Theorem 3.3 allows us to prove the following converse of Corollary~2.3.

\begin{theorem}
Let $D \subset \RR^m$, be a bounded domain in which the Dirichlet problem for
equation \eqref{Hh} is soluble. If $u \in C^2 (D) \cap C^0 (\overline D)$ has
$\mu$-panharmonic $M^\circ (u, \cdot, r)$ in $D_r$ for all $r \in (0, r_*)$, where
$r_* > 0$ is such that $D_{r_*} \neq \emptyset$, then $u$ is $\mu$-panharmonic in
$D$.
\end{theorem}

\begin{proof}
It is clear that every $x \in D$ belongs to each $D_r$ provided $r < \mathrm{dist}
(x, \partial D) / 2$, where $\mathrm{dist} (x, \partial D)$ is the distance from $x$
to $\partial D$. Let us fix some $r (x) \in (0, \mathrm{dist} (x, \partial D) / 2)$;
hence $\overline{B_{r (x)} (x)} \subset D_r$ for all described values of $r$. Since
the mean $M^\circ (u, \cdot, r)$ is $\mu$-panharmonic in $D_r$ for every such $r$,
Theorem~2.1 yields that
\[ M^\circ (M^\circ (u, \cdot, r), x, r (x)) = a^\circ (\mu \, r (x)) M^\circ (u,
x, r)
\]
In view of \eqref{im'} and \eqref{sm}, this can be written as follows:
\[ \frac{1}{\omega_m^2} \int_{S_1 (0)} \int_{S_1 (0)} u (x + r (x) y + r
z) \, \D S_z \, \D S_y = \frac{a^\circ (\mu \, r (x))}{\omega_m} \int_{S_1 (0)} u
(x +r y) \, \D S_y \,.
\]
Letting $r \to 0$ in this equality, we obtain that the identity 
\[ M^\circ (u, x, r (x)) = a^\circ (\mu \, r (x)) \, u (x)
\]
holds for every $x \in D$ with some $r (x)$ such that $\overline{B_{r (x)} (x)}
\subset D$. Now, Theorem~3.3 yields that $u$ is $\mu$-panharmonic in $D$.
\end{proof}

\subsection{Converse of the asymptotic mean value property}

The following converse of Proposition 2.1 generalizes the classical result obtained
by Blaschke \cite{Bla}, Priwaloff \cite{Pri} and Zaremba \cite{Za} for harmonic
functions.

\begin{theorem}
Let $D$ be a domain in $\RR^m$, $m \geq 2$, and let $u \in C^2 (D)$. If identity
\eqref{par} holds for every $x \in D$, then $u$ is $\mu$-panharmonic in $D$.

The assertion also holds with $M^\bullet (u, x, r)$ changed to $M^\circ (u, x, r)$
in \eqref{par}, provided the right-hand side term is changed to $\mu^2 u (x) / (2
m)$.
\end{theorem}

\begin{proof}
Let equality \eqref{par} hold; combining it and formula \eqref{vol} one obtains that
$u$ is $\mu$-pan\-harmonic in $D$. In the same way, \eqref{par} and \eqref{sph}
yield the second assertion when $M^\circ (u, x, r)$ stands in \eqref{par} instead of
$M^\bullet (u, x, r)$, whereas the right-hand side term is $\mu^2 u (x) / (2 m)$.
\end{proof}

{\centering \section{Characterizations of balls via panharmonic functions} }

It is worth mentioning first that analytic characterization of balls in the Euclidean
space $\RR^m$ by means of harmonic functions has a long history; it started in the
1960s, in the pioneering notes \cite{E}, \cite{ES}, and shortly afterwards the
following general result was obtained.

\begin{theorem}[Kuran \cite{K}]
Let $D$ be a domain (= connected open set) of finite (Lebesgue) measure in the
Euclidean space $\RR^m$ where $m \geq 2$. Suppose that there exists a point $P_0$ in
$D$ such that, for every function $h$ harmonic in $D$ and integrable over $D$, the
volume mean of $h$ over $D$ equals $h (P_0)$. Then $D$ is an open ball (disk when
$m=2$) centred at $P_0$.
\end{theorem}

Presumably, the paper \cite{HN} was the first one in which this theorem was referred
to as the property of harmonic functions inverse to the mean value identity for
balls. The term became widely accepted. A slight modification of Kuran's
considerations shows that his theorem is valid even if~$D$ is disconnected; see the
survey article \cite{NV}, p.~377, which also contains some improvements of Kuran's
theorem, and a discussion of its applications and of possible similar results
involving certain averages over $\partial D$, when $D$ is a bounded domain. It
occurs that panharmonic functions yield an analogous characterization of balls.

\subsection{Inverse mean value property: volume means}

The following result was recently proved in \cite{Ku5}; see also the brief note
\cite{Ku3}. Before giving its precise formulation, we give two definitions. If $D$
is a bounded domain and a function $f$ is integrable over $D$, then
\[ M^\bullet (f, D) = \frac{1}{|D|} \int_{D} f (x) \, \D x
\]
is the volume mean value of $f$ over $D$. Here and below $|D|$ is the domain's volume
(area if $D \subset \RR^2$). Also, we define a dilated copy of $D$: $D^r = D \cup
\left[ \cup_{x \in \partial D} B_r (x) \right]$. Thus, the distance from $\partial
D^r$ to $D$ is equal to $r$.

\begin{theorem}
Let $D \subset \RR^m$, $m \geq 2$, be a bounded domain, whose complement is
connected, and let $r$ be a positive number such that $|B_r| \leq |D|$. Suppose that
there exists a point $x_0 \in D$ such that for some $\mu > 0$ the mean value
identity $u (x_0) \, a^\bullet (\mu r) = M^\bullet (u, D)$ holds for every positive
function $u$, which is panharmonic in $D_r$, and $a^\bullet$ is defined in
\eqref{MM'}. If also $|D| = |B_r|$ provided $B_r (x_0) \setminus \overline D \neq
\emptyset$, then $D = B_r (x_0)$.
\end{theorem}

Prior to proving this theorem, let us consider some properties of the function
\begin{equation*}
U (x) = a^\circ (\mu |x|) \, , \quad x \in \RR^m , \label{U}
\end{equation*}
The properties of $a^\circ$ defined in \eqref{MM} show that this spherically
symmetric function monotonically increases from unity to infinity as $|x|$ goes from
zero to infinity.

Moreover, Poisson's integral for $I_\nu$ (see \cite{NU}, p.~223) implies that:
\begin{equation}
U (x) = \int_0^1 (1 - s^2)^{(m - 3)/2} \cosh (\mu |x| s) \, \D s \, .
\label{PI}
\end{equation}
This representation is easy to differentiate, thus obtaining that $U$ is panharmonic
in $\RR^m$. Since the formulae for $a^\circ$ and $a^\bullet$ are similar, Poisson's
integral allows us to compare these functions. In that way, the inequality
\begin{equation} 
[ U (x) ]_{|x| = r} > a^\bullet (\mu r) \label{aU} 
\end{equation} 
immediately follows.

\begin{proof}[Proof of Theorem 4.2.] 
Without loss of generality, we suppose that the domain $D$ is located so that $x_0$
coincides with the origin. Let us show that the assumption that $D \neq B_r (0)$
leads to a contradiction.

It is clear that either $B_r (0) \subset D$ or $B_r (0) \setminus \overline D \neq
\emptyset$ (the equality $|B_r| = |D|$ is assumed in the latter case), and we treat
these two cases separately. Let us consider the second case first, for which
purpose we introduce the bounded open sets
\[ G_i = D \setminus \overline{B_r (0)} \quad \mbox{and} \quad G_e = B_r (0)
\setminus \overline D \, ,
\]
whose nonzero volumes are equal in view of the assumptions about $D$ and $r$. The
volume mean identity for $U$ over $D$ can be written as follows:
\begin{equation}
|D| \, a^\bullet (\mu r) = \int_D U (y) \, \D y \, ; \label{1}
\end{equation}
here the condition $U (0) = 1$ is taken into account. Since formula \eqref{MM'} is
valid for $U$ over $B_r (0)$, we write it in the same way:
\begin{equation}
|B_r| \, a^\bullet (\mu r) = \int_{B_r (0)} U (y) \, \D y \, . \label{2}
\end{equation}
Subtracting \eqref{2} from \eqref{1}, we obtain
\begin{equation*}
0 = \int_{G_i} U (y) \, \D y - \int_{G_e} U (y) \, \D y > 0 \, .
\end{equation*}
Indeed, the difference is positive since $U (y)$ (positive and monotonically
increasing with~$|y|$) is greater than $[U (y)]_{|y| = r}$ in $G_i$ and less than
$[U (y)]_{|y| = r}$ in $G_e$, whereas $|G_i| = |G_e|$. This contradiction proves
the result in this case.

In the case when $B_r (0) \subset D$, a contradiction must be deduced when $B_r (0)
\neq D$, that is,  $|G_i| = |D| - |B_r| > 0$. Now, subtracting \eqref{2} from
\eqref{1}, we obtain
\begin{equation*}
( |D| - |B_r| ) \, a^\bullet (\mu r) = \int_{G_i} U (y) \, \D y > |G_i| \, [ U (y)
]_{|y| = r} \, , \label{3}
\end{equation*}
where the last inequality is again a consequence of positivity of $U (y)$ and its
monotonicity. This yields that $a^\bullet (\mu r) > [ U (y) ]_{|y| = r}$, which
contradicts \eqref{aU}. The proof is complete.
\end{proof} 

\begin{remark}
In Theorem 4.2, the domain $D$ is supposed to be bounded because it is easy to
construct an unbounded domain of finite volume in which $U$ is not integrable. Thus,
the boundedness of $D$ allows us to avoid imposing rather complicated restrictions
on the domain.

In the limit $\mu \to 0$, one obtains Laplace's equation from \eqref{Hh}, whereas
the assumption about $r$ becomes superfluous in this case. Hence, Theorem 4.2 turns
into an improved version of Kuran's theorem because only positive harmonic functions
are involved.

Furthermore, the integral $\int_D u (y) \, \D y$ can be replaced by the flux
$\int_{\partial D} \partial u / \partial n_y \, \D S_y$ in the formulation of
Theorem~4.2 provided $\partial D$ is sufficiently smooth; here $n$ is the exterior
unit normal. Indeed, we have
\[ \int_D u (y) \, \D y = \mu^{-2} \int_D \nabla^2 u \, (y) \, \D y = \mu^{-2}
\int_{\partial D} \partial u / \partial n_y \, \D S_y \, .
\]
These relations are used in \cite{Ku4}; see comments to Theorem 9 of that paper.
\end{remark}

\subsection{Characterization of balls via fundamental solutions of equation (1.1)}

A different approach to harmonic characterization of balls was developed by
Aharonov, Schiffer and Zalcman \cite{ASZ}. The origin of a rather unusual title of
their paper (potato kugel is a traditional dish of Jewish cuisine commonly served
for Shabbat) is explained in Zalcman's comment; see \cite{Sch}, p.~497. Namely,
these authors proved the following.

\begin{theorem}[ASZ, \cite{ASZ}]
 Let $D \subset \RR^3$ be a bounded open set. If the equality
\begin{equation*}
\int_D \frac{\D y}{|y - x|} = \frac{a}{|x|} + b \label{asz}
\end{equation*}
holds with suitable real constants $a$ and $b$ for every $x \in \RR^3 \setminus D$,
then $D$ is an open ball centred at the origin, $a = |D|$ and $b=0$.
\end{theorem}
Since $|y - x|^{-1}$ is a fundamental solution of the Laplace equation for $m=3$,
this theorem answers in the affirmative the following question posed to the authors
(see  \cite{ASZ}, p.~331):
\begin{quote}
Let $D$ be a solid, homogeneous, compact, connected ``potato'' in space, which
gravitationally attracts each point outside it as if all its mass were concentrated
at a single point [\dots] Must $D$ be spherical, i.e. a ball?
\end{quote}
There are various generalizations and improvements of this result. In particular,
the fol\-lowing one was obtained in the recent article \cite{CuLa}.

\begin{theorem}[Cupini, Lanconelli, \cite{CuLa}]
Let $D \subset \RR^m$, $m \geq 3$, be an open set such that $|D| < \infty$. If for
some $x_0 \in D$ the identity
\begin{equation*}
|D|^{-1} \int_D |y - x|^{2-m} \D y = |x_0 - x|^{2-m} \label{asz'}
\end{equation*}
holds for every $x \in \RR^m \setminus D$, then $D$ is an open ball centred at
$x_0$.
\end{theorem}

A similar result, to which we now turn, is valid for the potential
\begin{equation}
E_\mu^- (x, y) = \frac{\exp \{- \mu |x-y|\}}{|x-y|} \, , \quad \mu > 0 , \quad x \in
\RR^3 \setminus \{ y \} , \label{E-}
\end{equation}
rapidly decaying with the distance; it was proposed by Yukawa \cite{Y} to describe
a source of nuclear force located at $y \in \RR^3$. Following the paper \cite{ASZ},
we restrict our considerations to three dimensions, thus answering the quoted
question for the nuclear setting. Since there is another linearly independent
fundamental solution of \eqref{Hh}, namely,
\begin{equation}
E_\mu^+ (x, y) = \frac{\exp \{\mu |x-y|\}}{|x-y|} \, , \quad \mu > 0 , \quad x \in
\RR^3 \setminus \{ y \} , \label{E+}
\end{equation}
it must also be taken into account.

For every $r>0$ and arbitrary $x_0 \in \RR^3$, these fundamental solutions define
two families of integrable panharmonic functions
\[ B_r (x_0) \ni y \mapsto  E_\mu^\pm (y, x) \ \ \mbox{parametrised by} \ x \in \RR^3 
\setminus B_r (x_0) \, .
\]
For every element of these families the mean value property \eqref{MM'} yields that
\begin{equation}
a_3 (\mu r) \, E_\mu^\pm (x, x_0) = M^\bullet (E_\mu^\pm (\cdot, x), x_0, r) \, , \
\ a_3 (t) = \sqrt{2 \pi} \, I_{3/2} (t) / t^{3/2} \, ;
\label{MME+}
\end{equation}
the latter function is $a^\bullet (t)$ for $m=3$. In view of Theorem~4.4 and
identity \eqref{MME+}, we prove the following.

\begin{theorem}
Let $D \subset \RR^3$ be a bounded domain, whose complement is connected, and let $r
> 0$ be such that $|B_r| = |D|$. If the fundamental solutions \eqref{E-} and
\eqref{E+} satisfy the mean value identity
\begin{equation}
a_3 (\mu r) \, E_\mu^\pm (x, x_0) = M (E_\mu^\pm (\cdot, x), D) \label{MMED}
\end{equation}
for some $x_0 \in D$ and every $x \notin D$, then $D = B_r (x_0)$.
\end{theorem}

\begin{proof}
Since $E_\mu^+ (x, x_0)$ and $E_\mu^- (x, x_0)$ satisfy \eqref{MMED} for every $x
\notin D$, the same is true for every linear combination of these fundamental
solutions. In particular,
\[ |D| \, a_3 (\mu r) \frac{\sinh (\mu |x - x_0|)}{\mu |x - x_0|} =
\int_D \frac{\sinh (\mu |x - y|)}{\mu |x - y|} \, \D y \ \ \mbox{for every} \ x
\notin D .
\]
Moreover, this identity is valid throughout $\RR^3$, because real-analytic functions
of $x$ stand on both sides (a consequence of the fact that $z^{-1} \sinh z$ is an
entire function). Substituting $x=x_0$, we obtain
\[ |D| \, a_3 (\mu r) = \int_D \frac{\sinh (\mu |x_0 - y|)}
{\mu |x_0 - y|} \, \D y \, .
\]
Let us relocate, without loss of generality, the domain $D$ so that $x_0$ coincides
with the origin, which simplifies the identity to
\begin{equation}
|D| \, a_3 (\mu r) = \int_D U (y) \, \D y \, , \ \ \mbox{where} \ \ U (y) =
\frac{\sinh (\mu |y|)} {\mu |y|} \, , \label{1+}
\end{equation}
because this is the function \eqref{PI} with $m=3$.

On the other hand, the mean value property \eqref{MM'} is valid for $U$ over $B_r$:
\begin{equation}
|B_r| \, a_3 (\mu r) = \int_{B_r} U (y) \, \D y \, . \label{2+}
\end{equation}
If we assume that $D \neq B_r$, then $G_i = D \setminus \overline{B_r}$ and $G_e =
B_r \setminus \overline D$ are bounded open sets such that $|G_e| = |G_i| \neq 0$,
which follows from the assumptions made about $D$ and $r$. Then, subtracting
\eqref{2+} from \eqref{1+}, we obtain
\begin{equation*}
0 = \int_{G_i} U (y) \, \D y - \int_{G_e} U (y) \, \D y > 0 \, .
\end{equation*}
Indeed, the difference is positive since $U (y)$ (positive and monotonically
increasing with~$|y|$) is greater than $[U (y)]_{|y| = r}$ in $G_i$ and less than
$[U (y)]_{|y| = r}$ in $G_e$, whereas $|G_i| = |G_e|$. The obtained contradiction
proves the result.
\end{proof}

\begin{remark}
The final part of this proof repeats literally the argument used in the proof of
Theorem~4.2.
\end{remark}

\vspace{-3mm}

{\centering \section{Relations between harmonic and panharmonic functions} }

The motivation to consider relations between harmonic and panharmonic functions
comes from the theorem on subharmonic functions published by F.~Riesz \cite{R} in
1930. It establishes the decomposition of such a function into the sum of a harmonic
function and a Newtonian potential. (The result was proved by Riesz for functions of
two variables, whereas the general case can be found in \cite{HK}, Section~3.5.) It
occurs that panharmonic and subharmonic functions have a lot in common; see
Corollary~2.1. Therefore, it is reasonable to apply methods developed for
subharmonic functions in studies of panharmonic ones.

\vspace{3mm}

\subsection{Properties of positive panharmonic functions}

It is worth to recall the Riesz decomposition theorem for subharmonic functions
(see, for example, \cite{HK}, Theorem~3.9). 

\begin{theorem}
If $u$ is subharmonic in a domain $D \subset \RR^m$, $m \geq 2$, then there exists
a unique Borel measure $\mathbf m$ in $D$ such that for any compact set $K \subset
D$
\begin{equation}
u (x) = \int_K E_m (x - y) \, \D \mathbf m (y) + h (x) , \quad x \in \mathrm{int} K
, \label{Ri}
\end{equation}
where $\mathrm{int} K$ is the interior of $K$ and $h$ is harmonic there.
\end{theorem}

Here $E_m (x - y)$ is the fundamental solution of the Laplace equation:
\[ E_m (x - y) = \left[ (2 - m) \, \omega_m |x - y|^{(m-2)} \right]^{-1} \ \ 
\mbox{when} \ m \geq 3 \, ,
\]
whereas $E_ (x - y) = (2 \pi)^{-1} \log |x - y|$.

\begin{remark}
It follows from Treves' considerations (see \cite{T}, pp.~288--289) that if $u \geq
0$ is subharmonic in a bounded domain $D$, then formula \eqref{Ri} holds with $K$
changed to $D$, whereas $\D \mathbf m (y) = \nabla^2 u (y) \, \D y$ and $h$ is the
positive least harmonic majorant of $u$ in~$D$; for its definition see also
\cite{AG}, p.~79.
\end{remark}

Now we are in a position to formulate the following.

\begin{theorem}
Let $u \geq 0$ be $\mu$-panharmonic in a domain $D \subset \RR^m$, $m \geq 2$, then
\eqref{Ri} takes the following form:
\begin{equation}
h (x) = u (x) - \mu^2 \int_K E_m (x - y) \, u (y) \, \D y , \quad x \in \mathrm{int}
K . \label{Rie}
\end{equation}
Here $K \subset D$ is a compact set and $h$ is harmonic in $\mathrm{int} K$.

If $D$ is bounded and, besides, $u \in C^0 (\bar D)$, then \eqref{Rie} is valid in
the whole $D$ with the integral over $D$, whereas $h \geq 0$ is the least harmonic
majorant of $u$ in $D$.
\end{theorem}

\begin{proof}
According to Corollary 2.1, $u$ is subharmonic in $D$, and so it has the Riesz
decomposition \eqref{Ri}. Applying the Laplacian to both sides of \eqref{Ri} and
taking into account equation \eqref{Hh} on the left-hand side and using the
harmonicity of $h$ and the definition of $E_m$ on the right, we see that $\mathbf m$
is proportional to the Lebesgue measure with the coefficient $\mu^2 u$ (cf.
Remark~5.1). Now \eqref{Rie} follows by rearranging.

The second assertion is obvious, whereas the last one is a consequence of the
considerations mentioned in Remark~5.1.
\end{proof}

Our next result involves mean values over a domain as well as over its boundary. In
this case, one can hardly expect an identity analogous to \eqref{MMtil} to be valid
for panharmonic functions in a domain distinct from a ball. Indeed, Bennett \cite{B}
proved the following.

\begin{theorem}
Let $D \subset \RR^m$ be a bounded domain with sufficiently smooth boundary.~If
\[ |D|^{-1} \int_D h (y) \, \D y = |\partial D|^{-1} \int_{\partial D} h (y) \, \D S_y
\]
for every $h \in C^2 (\bar D)$ harmonic in $D$, then $D$ is an open ball.
\end{theorem}

A similar conjecture for panharmonic functions based on identity \eqref{MMtil} is
made in \cite{Ku5}. At the same time, an inequality holds between the mean values of
nonnegative panharmonic functions in a bounded domain under a suitable assumption
about its boundary.

\begin{proposition}
Let $D \subset \RR^m$ be a bounded domain satisfying the exterior sphere condition
uniformly on $\partial D$. Then there exists a constant $c \in [1, \infty)$,
depending on $D$ and~$\mu$, and such that
\begin{equation}
|D|^{-1} \int_D u (y) \, \D y \leq c \, |\partial D|^{-1} \int_{\partial D} u (y) \,
\D S_y \label{FM}
\end{equation}
for every nonnegative panharmonic function $u \in C^0 (\bar D)$.
\end{proposition}

In view of Corollary 2.1, inequality \eqref{FM}, like Theorem~5.2, is a consequence
of the corresponding theorem proved for subharmonic functions; see \cite{FM},
p.~195. Moreover, if $D$ is a ball $B_r$ (there is no need to specify its center),
then equality takes place in \eqref{FM} with
\begin{equation*}
c = \frac{a^\bullet (\mu r)}{a^\circ (\mu r)} = \frac{m I_{m/2} (\mu r)}{\mu r
I_{(m-2)/2} (\mu r)} < \frac{m}{\mu r} \, . \label{18}
\end{equation*}
Here the equalities follow from identity \eqref{MMtil} and formulae \eqref{MM} and
\eqref{MM'}, whereas the inequality is a consequence of the definition of $I_\nu$.
This not only demonstrates that $c$ depends on $\mu$, but also improves
Proposition~5.1 for balls provided $\mu r > m$. It occurs that $c$ can be
arbitrarily small when either $r$ ($\mu$ fixed) or $\mu$ ($r$ fixed) is sufficiently
large (or both are sufficiently large).

\subsection{Characterization of panharmonic functions}

Let $D$ be a bounded domain in $\RR^m$, $m \geq 3$; for $u \in L^2 (D)$ we define
the operator:
\[ (T u) (x) = \int_D E_m (x - y) \, u (y) \, \D y , \quad x \in D .
\]
Its symmetric kernel is positive after dropping the negative coefficient and it has
well-known properties; for example, $T$ is compact in the Banach space $C^0 (\bar
D)$ (see, for example, \cite{M}, Chapter~7).

In terms of this operator, the second assertion of Theorem 5.2 admits the following
interpretation: $I - \mu^2 T$ (as usual, $I$ stands for the identity operator) maps
the cone of nonnegative $\mu$-panharmonic functions into the cone of nonnegative
harmonic functions within the Banach space $C^0 (\bar D)$.

Let us consider whether there exists an inverse mapping: harmonic $\mapsto$
$\mu$-panharmonic functions. To this end we introduce the integral equation
\begin{equation}
u (x) - \lambda (T u) (x) = h (x) , \quad x \in D , \ \ \lambda \in \RR ,
\label{21}
\end{equation}
where $u, h \in L^2 (D)$. This is a natural setting because the operator $T$ has a
weakly singular kernel, and so is compact and self-adjoint, whereas $-T$ is a
positive operator in this space. We recall that these properties of $T$ imply that
it has a sequence $\{\lambda_n\}_1^\infty$ of characteristic values each having a
finite multiplicity; moreover, these values are real negative numbers such that
$|\lambda_n| \to \infty$ as $n \to \infty$. Finally, if $\lambda \neq \lambda_n$ for
$n = 1,2,\dots$ (in particular, if $\lambda > 0$), then for any $h \in L^2 (D)$
equation \eqref{21} has a unique solution $u \in L^2 (D)$, which can be represented
by virtue of the resolvent kernel. Taking into account these facts, we formulate and
prove the following assertion, in which $C^{0, \alpha} (\bar D)$ stands for the
Banach space of functions that are H\"older continuous with exponent $\alpha \in (0,
1)$.

\begin{theorem}
Let $D$ be a bounded Lipschitz domain in $\RR^m$, $m \geq 3$, and let $\lambda =
\mu^2 > 0$ in equation \eqref{21}. If $h \in C^{0, \alpha} (\bar D)$ is harmonic in
$D$, then a unique solution $u$ of this equation belongs to $C^{0, \alpha} (\bar D)$
and is $\mu$-panharmonic in $D$.
\end{theorem}

\begin{proof}
It is a classical result (see, for example, \cite{M}, Theorem~8.6.1) that an
$L^2$-solution of a weakly singular integral equation (it exists in our case) is in
$C^{0} (\bar D)$ provided the right-hand side term has this property. However, the
continuity of $u$ does not guarantee the existence of second derivatives of the
Newtonian potential~$T u$. Let us establish their existence under the assumptions
made in the theorem.

Since $h \in C^{0, \alpha} (\bar D)$, the solution $u$ has the same property.
Indeed, writing the equation as follows
\begin{equation}
u = \mu^2 T u + h \, , \label{22}
\end{equation}
we see that both terms on the right are in $C^{0, \alpha} (\bar D)$, because this is
a consequence of the following result (see \cite{V}, Lemma 2.3). If $u \in L^\infty
(D)$, then $T u \in C^{0, 1} (\bar D)$, that is, $T u$ is Lipschitz continuous. Now,
another classical result (see \cite{GT}, Lemma~4.2) yields that $T u \in C^{2} (D)$
and $\nabla^2 (T u) = u$. Furthermore, relation \eqref{22} implies that $u \in C^{2}
(D)$ since $h$ is harmonic in $D$. Then, applying the Laplacian to both sides of
\eqref{22}, we obtain that $u$ is $\mu$-panharmonic in $D$.
\end{proof}

In other words, for every $\mu^2 > 0$ there exists the bounded operator
\[ (I - \mu^2 T)^{-1} : L^2 (D) \to L^2 (D) \, ,
\]
which maps each harmonic in $D$ function from $C^{0, \alpha} (\bar D)$ to a
$\mu$-panharmonic function belonging to the same H\"older space. It is not clear
whether the range of this operator comprises the whole set of $\mu$-panharmonic
functions. At the same time, Theorem~5.2 yields that every nonnegative function
belonging to this set has a pre-image in the cone of nonnegative harmonic functions.
Hence, all nonnegative $\mu$-panharmonic functions are in the operator's range.

\vspace{-14mm}

\renewcommand{\refname}{
\begin{center}{\Large\bf References}
\end{center}}
\makeatletter
\renewcommand{\@biblabel}[1]{#1.\hfill}
\makeatother

\end{document}